\documentclass[12pt, reqno, a4paper]{amsart}
\usepackage{amsmath,amsthm,amsfonts,amssymb,enumerate,amscd,color,bm}
\numberwithin{equation}{section}
\usepackage[T1]{fontenc}
\usepackage{mathrsfs}

\theoremstyle{plain}
\newtheorem{theo}{Theorem}[section]
\newtheorem{coro}[theo]{Corollary}
\newtheorem{lem}[theo]{Lemma}
\newtheorem{prop}[theo]{Proposition}

\theoremstyle{definition}
\newtheorem{defi}{Definition}[section]

\theoremstyle{remark}
\newtheorem{rem}{Remark}[section]

\newcommand{\ii}{\infty}

\renewcommand{\v}{\varphi}

 \newcommand{\al}{\alpha}
 \newcommand{\be}{\beta}
\newcommand{\la}{\lambda}

\newcommand{\si}{\sigma}

\newcommand{\st}{such that }
\newcommand{\Th}{Theorem }
\newcommand{\f}[2]{\frac{#1}{#2}}

\renewcommand{\d}{\partial}
\newcommand{\R}{\mathbb{R}}
\newcommand{\h}{\hspace*{5mm}}

\renewcommand{\i}{{\rm i}}

\newcommand{\de}{Definition }

\newcommand{\fu}{function }

\newcommand{\fus}{functions }

\newcommand{\pb}{problem }

\newcommand{\pr}{Proposition }

\newcommand{\so}{solution }
\newcommand{\sos}{solutions }
\renewcommand{\th}{theorem }

\renewcommand{\t}[1]{\ \mbox{ #1}\ }

\newcommand{\rn}{\!\!}
\newcommand{\vp}{,\h \forall }

\newcommand{\dif}{\,{\rm d}}
\newcommand{\ddt}{\f{{\rm d}}{{\rm d}t}}

\newcommand{\ld}{L^2(\rd)}
\newcommand{\lm}{L^2(X, \mu)}
\newcommand{\n}[1]{\| #1 \|}

\def\db{{\bf D}_{0, t}}
\def\rd{{\mathbb R^d}}
\renewcommand{\d}{{\bf D}^\alpha_{-\ii, t}}
\newcommand{\ds}{{\bf D}^\alpha_{-\ii, \si}}
\newcommand{\dm}{{\bf D}^\alpha_{t, \ii}}
\newcommand{\dms}{{\bf D}^\alpha_{\si, \ii}}

\newcommand{\C}{\mathbb{C}}

\def\F{\mathcal F}

\renewcommand{\L}{\mathcal{L}}

\newcommand{\U}{\mathcal{U}}

\def\S{\mathcal S}

\def\dt{{\, \mathrm d}t}

\def\dy{{\mathrm d}y}

\advance \vsize by 1cm
\begin{document}
\title[]{Time Fractional Schrödinger Equation}
\author{Hassan  Emamirad}
\author{Arnaud Rougirel}
\thanks{This research was in part supported by a grant from IPM \# 91470221}

\address{Hassan  Emamirad\newline
School of Mathematics, Institute for Research in Fundamental Sciences (IPM), P.O. Box 19395-5746, Tehran, Iran\newline
Laboratoire de  Math\'{e}matiques, Universit\'{e} de Poitiers \& CNRS. Téléport  2, BP 179, 86960 Chassneuil du Poitou Cedex, France}
\email{emamirad@ipm.ir}
\email{emamirad@math.univ-poitiers.fr}

\address{Arnaud Rougirel\newline
Laboratoire de  Math\'{e}matiques, Universit\'{e} de Poitiers \& CNRS. Téléport  2, BP 179, 86960 Chassneuil du Poitou Cedex, France}

\email{rougirel@math.univ-poitiers.fr}


 \date{\today}

\begin{abstract}
We propose a time fractional extension of the Schrödinger equation that keeps the main mechanical and quantum properties of the classical Schrödinger equation. This extension is shown to be equivalent to another well identified time first order PDE with fractional hamiltonian.    
\end{abstract}
\maketitle
\thispagestyle{empty}
\section{Introduction}
We would like to address the issue of what are suitable time fractional extensions of the \emph{Schrödinger equation}. Our approach consists in, first, selecting some (important) properties of the Schrödinger equation; secondly, finding a time fractional equation which conserves these properties.

We have selected two properties: (i) the conservation of the $L^2$-norm of wave functions. That point is central w.r.t. the probabilistic character of quantum mechanics objects.  (ii) The dynamics of the Schrödinger equation, i.e.\! the time reversibility of \sos which is related to the fact that the related \emph{solution operator} is a group. Indeed, in any reasonable mechanical theory, 
trajectories of autonomous systems must be described by equations generating a 
group.

It is clear that the use of time fractional operators as $\db^\al$, do not allow to match the above requirements. We refer to Remark \ref{r4.1} below for details, 
and to \cite{Die} or \cite{SKM} for more information on fractional derivatives. Regarding the dynamics of time fractional equations, we refer to \cite{ER}. 

In this paper, the real number $\al$ is always supposed to range between $0$ and $1$.
 
In order to recover the above properties, we consider time fractional operators with lower bound $-\ii$, i.e. $\d$. More precisely, if, for simplicity, we restrict our attention to the Hamiltonian $\hat{H}:=-\Delta$, then we 
will show that the \pb
\begin{equation} \label{1.2}
    \d u = -\i^\al\Delta u,\h u(0)=v,
\end{equation}
admits a unique \so $u$ in some suitable function space. \Th \ref{t4.1} states 
that the \so to \eqref{1.2} is given by 
$$
    u(t)= {\rm e}^{\i t (-\Delta)^{1/\al}} v,
$$
hence the $L^2$-norm of $u$ is conserved and the \so operator of \eqref{1.2} is 
a group. Then, according to Corollary \ref{c4.1}, \eqref{1.2} is equivalent to
\begin{equation} \label{1.4}
    \ddt u = \i(-\Delta)^{1/\al} u,\h u(0)=v.
\end{equation}

The coefficient $\i^\al$ in \eqref{1.2} is not surprising since we know from the work of Naber, in \cite{Na}, that the operator $\db + \i\Delta$ has a parabolic 
behavior (see  Remark \ref{r4.1}).

In this paper, we consider more generally abstract \emph{time fractional Schrödinger equations} of the form
\begin{equation} \label{1.6}
    \ddt u = \i^\al A u,
\end{equation}
where $A$ is a positive self-adjoint operator.

The outline is as follows. The forcoming section is dedicated to preliminaries regarding fractional derivatives, in particular time fractional \emph{weak} derivatives. In section 3, we solve \eqref{1.6} in the case where the underlying Hilbert space in $\R$. That turns out to be the corner stone of our work since, in section 4, we solve \eqref{1.6} by diagonalisation using \emph{spectral theory}.   
\section{Fractional derivatives with lower bound $-\ii$}\label{s2}
We start with the convolution of \fus defined on $\R$, with the 
\emph{fractional kernel}, whose definition is as follow.
%
%
\begin{defi}\label{d2.2} For $\be\in(0, \ii)$, let us denote by $g_\be$ the \fu of $L^1_{loc}([0,\ii))$ defined for a.e. $t>0$ by
$$
        g_\be(t) = \f{1}{\Gamma(\be)}t^{\be-1}.
$$
\end{defi}
%
%
%
Let $X$ be a complex Banach space with norm $\n{\cdot}$. 
%
%
\begin{prop}  \label{p2.1} Let $\al\in (0, 1)$ and 
$u\in L^1(\R; X)\cap L^\ii(\R; X)$. Then, for all $t\in\R$, the \fu
$$
    (-\ii, t)\to X,\h y \mapsto g_{\al}(t-y)u(y)
$$
is integrable on $(-\ii, t)$ and
$$
    \sup_{t\in\R}\int_{-\ii}^t g_{\al}(t-y)\n{u(y)} \dy <\ii.
$$
\end{prop}
%
%
\begin{proof}
For all $t\in\R$, we have
\begin{align*}  
    \int_{-\ii}^t g_{\al}(t-y)\n{u(y)} \dy  &\le
    \int_{-\ii}^{t-1}g_{\al}(t-y)\n{u(y)} \dy
    +\int_{t-1}^{t}g_{\al}(t-y)\n{u(y)} \dy \\
     & \le g_{\al}(1)\n{u}_{L^1(\R; X)}+ g_{\al+1}(1)\n{u}_{L^\ii(\R; X)}.
\end{align*}
\end{proof}
Under the assumptions and notation of \pr \ref{p2.1}, we put, for all $t\in\R$,
$$
      g_{\al}* u(t):= \int_{-\ii}^t g_{\al}(t-y)u(y) \dy
                    = \int_{0}^\ii g_{\al}(y)u(t-y) \dy.
$$
Also, we define
$$
      g_{\al}*' u(t):= \int^{\ii}_t g_{\al}(y-t)u(y) \dy
                    = \int_{0}^\ii g_{\al}(y)u(t+y) \dy.
$$
In some sense, these convolutions are adjoints. More precisely, we have the following result, whose easy proof is left to the reader.

%
%
\begin{prop}  \label{p2.2} Let $\al\in (0, 1)$,  
$u\in L^1(\R; X)\cap L^\ii(\R; X)$ and $\psi\in L^1(\R)\cap L^\ii(\R)$. Then
\begin{equation} \label{2.4}
    \int_\R g_\al * u(t)\, \psi(t)\dt =\int_\R u(t)\, g_\al *' \psi(t)\dt.     	
\end{equation}
\end{prop}
%
%
%

Then we may give the following definition of \emph{fractional derivatives}.
%
%
%
\begin{defi}\label{d2.3} Let $\al\in (0, 1)$ and 
$u\in L^1(\R; X)\cap L^\ii(\R; X)$.  We say that 
$u$ \emph{admits a (forward) derivative of order $\al$ in} 
$L^\ii(\R; X)$ if  
$$
      g_{1-\al}*u \in W^{1,\ii}(\R; X).
$$  
In this case, its \emph{(forward) derivative of order $\al$} is the \fu of 
$L^\ii(\R; X)$ defined by
$$
    \d u := \f{{\rm d}}{{\rm d}t}\big\{
    g_{1-\al}*u \big\}.
$$ 
\end{defi}
%
%
%
\begin{defi}\label{d2.4} Let $\al\in (0, 1)$ and 
$u\in W^{1,1}(\R; X)\cap W^{1,\ii}(\R; X)$. Then we say that 
$u$ \emph{admits a backward derivative of order $\al$ in} 
$L^\ii(\R; X)$ if
$$
      g_{1-\al}*'\f{{\rm d}}{{\rm d}t} u \in L^\ii(\R; X).
$$
In this case, its \emph{backward derivative of order $\al$} is the \fu of 
$L^\ii(\R; X)$ defined by
$$
    \dm u := g_{1-\al}*'\f{{\rm d}}{{\rm d}t} u.
$$ 
\end{defi}
Let $\S(\R)$ denote the Schwartz space of rapidly decreasing complex \fus defined on $\R$.
%
%
\begin{prop}  \label{p2.3} 
Let $\al\in (0, 1)$, $u\in L^1(\R; X)\cap L^\ii(\R; X)$ and $\v\in \S(\R)$. 
Assume that $u$ 
admits a fractional derivative of order $\al$ in $L^\ii(\R; X)$. Then
\begin{equation} \label{2.6}
    \int_\R \d u(t)\v(t)\dt =
    -\int_\R u(t)\dm\v(t)\dt.
\end{equation}
\end{prop}
%
%
%
\begin{proof}
Integrate by parts and use \pr \ref{p2.2}.
\end{proof}
We will now introduce \emph{fractional derivatives in the sense of distributions}. The following result makes possible such a definition.
%
%
\begin{prop}  \label{p2.4} 
Let $\al\in (0, 1)$ and $\v\in \S(\R)$. Then $\dm\v\in L^1(\R)$ and
\begin{equation} \label{2.8}
    \n{\dm\v}_{L^1(\R)} \le C\big(\n{\v}_{L^1(\R)} + \n{\v'}_{L^1(\R)}\big),
\end{equation}
where $\v':=\ddt\v$ and the constant $C$ depends only on $\al$.
\end{prop}
%
%
%
\begin{proof} For any fixed time $t$ in $\R$, let $v(y):=\v(y+t)-\v(t)$. 
Integrating by part, we get
$$
       \al\int_0^\ii y^{-\al-1}v(y)\dy =
       \int_0^\ii y^{-\al}\v'(y+t)\dy.
$$
Hence

$$
     \dm \v(t)=\f{1}{|\Gamma(-\al)|}\int_0^\ii
     y^{-\al-1}\big(\v(y+t)-\v(t)\big)\dy,
$$
and
\begin{align*}  
     |\Gamma(-\al)\dm \v(t)|& \le 
    \int_0^1 y^{-\al} \big| \f{\v(y+t)-\v(t)}{y}\big| \dy &+&
    \int_1^\ii y^{-\al-1} | \v(y+t)-\v(t)| \dy \\
    & =:  \h \h\h I_1(t)\h &+&\h\h\h I_2(t),
\end{align*}
with obvious notation for $I_1(t)$ and $I_2(t)$.

Let us show that $I_1$ is integrable on $\R$. For, since
$$
      \f{\v(y+t)-\v(t)}{y} = \int_0^1 \v'(ys+t)\dif s \vp y \in (0, 1],
$$
we derive 
$$
     \int_{\R}|I_1(t)|\dt \le \int_0^1 y^{-\al}\dy \int_0^1\dif s
                              \int_{\R}|\v'(ys+t)|\dt 
                              = \f{\n{\v'}_{L^1(\R)}}{1-\al}.
$$

Regarding $I_2(t)$, we have
\begin{align*}  
      \int_{\R}|I_2(t)|\dt  & \le
    \int_1^\ii y^{-\al-1}\int_\R |\v(y+t)| + |\v(t)|\dt \\
     & = \f{2}{\al}\n{\v}_{L^1(\R)}.
\end{align*}
%
\end{proof}
That estimate allows us to define fractional derivatives in the sense of 
distributions. Indeed, \eqref{2.8} shows that, for each 
$u\in L^\ii(\R; X)$, the linear map
$$
    \S(\R)\to X,\h  \v\mapsto -\int_\R u(t)\dm \v(t)\dt
$$
is a tempered distribution. The set of tempered distributions with values in $X$ is denoted by $\S'(\R; X)$. That allows us to set the following definition. 
%
%
%
%
%
\begin{defi}\label{d2.5} Let $\al\in (0, 1)$, $X$ be a complex Banach space and 
$u\in L^\ii(\R; X)$. Then the \emph{weak derivative of} $u$ is the $X$-valued 
tempered distribution, denoted by $\d u$, and defined, for all 
$\v\in \S(\R)$, by
$$
    \langle \d u, \v \rangle =  -\int_\R u(t)\dm \v(t)\dt.
$$
\end{defi}
If we want to highlight the duality taking place in the above bracket, we will write
$$
     \langle \d u, \v \rangle_{\S'(\R; X), \S(\R)}
     \h\mbox{or}\h \langle \d u, \v \rangle_{\S'(\R; X)}
$$
instead of $ \langle \d u, \v \rangle$.

In this paper, we will need to compute $\d u$ for bounded \fus $u$. For such \fus\rn, the integral
$$
        \int_{-\ii}^t g_{1-\al}(t-y)u(y) \dy
$$
is, in general, not absolutely convergent. That point turns out to be a major drawback; and we will explain the reason in the sequel.

Works on fractional calculus with lower terminal $-\ii$ and not absolutely convergent integrals, go back, at least to 1938 with the smart paper \cite{Lo}.  The difficulties when dealing with non 
absolutely convergent integral are illustrated by \Th 5 in \cite{Lo}, which states, roughly speaking that each $u\in L^\ii(\R)$ \st
\begin{equation} \label{2.10}
    \sup_{t\in\R}\big| \int_0^t u(y)\dy \big| < \ii,
\end{equation}
satisfies, for all $t\in\R$, the following (hard to prove) identity
$$
              g_{1-\al} * (g_\al*u)(t) = g_{1-\al} * (g_\al*u)(0) 
                                         + \int_0^t u(y)\dy.
$$

That result is quiet surprising since, extending $g_\al$ by $0$ outside of $(0, \ii)$, we have
\begin{equation*} 
       g_{1-\al} * g_\al(t)= 
	\begin{cases} 
		1& \t{if } t>0 \\
		0& \t{if } t<0
	\end{cases};
\end{equation*}
so that the convolution
$$
       (g_{1-\al} * g_\al)*u(t) = \int_{-\ii}^t u(y)\dy 
$$
is not defined in general. 

In this paper, we deal with bounded \fus satisfying typically \eqref{2.10}. However, Love's approach (see also \cite{BM}) allowed us to obtain only partial results. That is, we must assume that $\al=1/n$ for some positive integer $n$.
%
%
%
%
%
%
%
%
\section{Time fractional equations in $\R$} We will use the framework of Section \ref{s2} with $X=\C$. For $a\in\R$ and $\al\in (0, 1)$, we consider this equation:
%
\begin{equation}\label{3.2}
  \left\{
    \begin{aligned}
  & \mbox{\rm Find }u\in L^\ii(\R) \t{\st}\\
  & \d u = (\i a)^\al u,\h\t{in }\S'(\R).
\end{aligned}\right.
\end{equation}
Of course, in \eqref{3.2}, $\d u $ is understood in the sense of \de \ref{d2.5}.

For $\psi\in\ld$, we denote by $\F \psi$ or $\hat{\psi}$ its 
Fourier transform. More precisely, for $\v\in L^1(\R)$, we have
\begin{equation} \label{3.3}
    \hat{\v}(\si) = c_{\F}\int_\R {\rm e}^{-\i \si t}\v(t)\dt,\h
     \F^{-1}\v(t) =  c_{\F^{-1}}\int_\R {\rm e}^{\i t\si}\v(\si)\dif\si,
\end{equation}
where $c_{\F}$, $c_{\F^{-1}}$ denote reader's favorite constants of normalization, whose product equals $\f{1}{2\pi}$.
%
%
\begin{rem}\label{r3.1} Solving the fractional equation   
\begin{equation} \label{3.12}
    \d u = (\i a)^\al u
\end{equation}
is a touchy business. 

First, it is known that the \fu $u_1:t\mapsto \exp(\i a t)$ satisfies
$$
       \int_{0}^\ii g_{1-\al}(y)u_1(t-y) \dy = (\i a)^{\al-1}u_1(t).
$$
Thus, $u_1$ solves \eqref{3.12} in some ``classical sense''. However, the above integral is not absolutely convergent. Thus it is not clear how to get uniqueness 
for \eqref{3.12} in a space containing $u_1$ (typically $C_{\rm b}(\R)$). This is the reason why we use time fractional derivatives in the sense of 
distributions to solve \eqref{3.12}. 

But then, existence for \eqref{3.2} is tricky since, we have not been able to show directly that $u_1$ solves \eqref{3.2}. Indeed, Fubini's \Th being not applicable, the path from
$$
    \int_\R {\rm e}^{\i at}\dm \v(t)\dt = 
    \int_\R {\rm e}^{\i at}\int^{\ii}_t g_{1-\al}(y-t)\v'(y) \dy
$$
to 
$$
     \int_\R \v'(y) \dy \int_{-\ii}^y 
     {\rm e}^{\i at} g_{1-\al}(y-t)\dt
$$
is not clear. We overcome this difficulty by using Fourier transforms.

The last difficulty is technical and concerns the use of Fourier transforms. 
More precisely, starting from \eqref{3.12}, we obtain formally by Fourier transform, 
$$
        (\i \si)^\al \hat{u} = (\i a)^\al \hat{u}.  
$$
However, if $u=u_1$ then 
$$
        (\i \si)^\al \hat{u} = \textstyle{\f{1}{c_{\F^{-1}}}}(\i \si)^\al\delta_a,
$$
which has no meaning since the \fu $\si\mapsto (\i \si)^\al$, being 
not smooth at $\si=0$, cannot multiply $\delta_a$, the Dirac mass at the point $a$.

In order to overcome this difficulty, we first determine the support of 
$\hat{u}$, which turns out to be discrete, as expected; and then compute $u$ by inverse Fourier transform. 
\end{rem}
%
%
%
%
%
\begin{theo}\label{t3.1} Let  $a\in\R$, $\al\in (0, 1)$ and $u\in L^\ii(\R)$. Then $u$ is a \so of \eqref{3.2} iff there exists $k\in\C$ \st
$$
      u(t) = k \exp(\i at),\h\t{a.e. }t\in\R.
$$
\end{theo}
The following lemma will be usefull in the proof of \Th \ref{t3.1}.
%

%
%
\begin{lem}  \label{l3.1} 
Let $\v\in\S(\R)$. Then, for all $\si\in\R$, 
\begin{align}     
   \F\big(\d\v\big)(\si)  & = (\i \si)^\al \hat{\v}(\si)  \label{3.4}\\	
   \dms \hat{\v}(\si) & = -\F\big( (\i t)^\al\v \big) (\si). \label{3.6}
\end{align}
\end{lem}
%
%
\begin{proof} For the proof of \eqref{3.4}, we refer to \cite[Section 7.1]{SKM}. We establish \eqref{3.6} by a duality argument. Indeed, for all $\v$ and 
$\psi$ in $\S(\R)$, there holds, by \pr \ref{p2.3},
\begin{align*}  
      	\int_\R \dms \hat{\v}(\si) \psi(\si)\dif\si 
        & = -\int_\R \hat{\v}(\si) \ds\ \psi(\si)\dif\si \\
        & = -\int_\R \v(\si) \F\big(\d\ \psi\big)(\si)\dif\si,
\end{align*}
by Plancherel's formula. Thus with \eqref{3.4} and Plancherel's formula once again, we get
\begin{align*}  
      	\int_\R \dms \hat{\v}(\si) \psi(\si)\dif\si 
        & = -\int_\R \v(\si) (\i \si)^\al \hat{\psi}(\si) \dif\si\\
        & = -\int_\R \F\big( (\i t)^\al\v \big)(\si) \psi(\si) \dif\si.
\end{align*}
That proves \eqref{3.6}.
       
\end{proof}
%
%
%
\begin{proof}[Proof of \Th \ref{t3.1}] (i) Let us show that
the \fu $u:t \mapsto{\rm e}^{\i at}$ is \so to \eqref{3.2}. For, we compute 
\begin{align*}  
     \langle \F\big(\d u\big), \v \rangle_{\S'(\R), \S(\R)}
     & = \langle \d u, \hat{\v} \rangle\\
     & = -\int_\R u(\si) \dms \hat{\v}(\si)\dif\si\\
     & = \int_\R {\rm e}^{\i a\si} \F\big( (\i t)^\al\v \big)(\si)\dif\si,
\end{align*}
by \eqref{3.6}. Besides, the Fourier transform of $t\mapsto (\i t)^\al\v(t)$ 
belongs to $L^1(\R)$ by \eqref{3.6} and Proposition \ref{p2.4}. Then the 
latter integral is equal, by inverse Fourier transform (see \eqref{3.3}), to
\begin{align*}  
      	\f{1}{c_{\F^{-1}}} (\i a)^\al \v(a)
      = (\i a)^\al \langle \f{\delta_{a}}{c_{\F^{-1}}}, \v \rangle
      &= (\i a)^\al \langle \F\big({\rm e}^{\i at} \big), \v \rangle\\
      &= (\i a)^\al \langle \F u , \v \rangle.	
\end{align*}
Thus by invertibility of the Fourier transform in $\S'(\R)$, we deduce that 
$t\mapsto {\rm e}^{\i at}$ is \so to \eqref{3.2}.

(ii) Conversely, let us show that each \so $u$ to \eqref{3.2} is proportional 
to some exponential \fu\rn. Computing as above, we get, for each test-\fu $\v$ in $\S(\R)$,
$$
     \langle \F\big(\d u\big), \v \rangle_{\S'(\R), \S(\R)}=
     \int_\R u(\si) \F\big( (\i t)^\al\v \big)(\si)\dif\si.
$$
Thus, since $u$ solves \eqref{3.2},
$$
    0= \int_\R u(\si) \F\big( \{(\i t)^\al - (\i a)^\al\}\v \big)(\si)\dif\si.
$$

Now, set 
\begin{equation*}
  I_a :=
	\begin{cases} 
		\{0\}  & \t{if }a=0  \\
		\{0,a\}& \t{if }a \not= 0
	\end{cases}.
\end{equation*}
Then, for each $\psi\in \S(\R)$ with support, denoted by ${\rm supp}\psi$, in 
$\R\setminus I_a$, we may find a 
\fu $\v\in \S(\R)$ so that
$$
      \psi(t)=\big((\i t)^\al - (\i a)^\al\big)\v (t)\vp t\in {\rm supp}\psi.
$$

There results that
$$
      \langle \hat{u}, \psi \rangle_{\S'(\R), \S(\R)}= 0.
$$
Thus the support of the distribution $\hat{u}$ lies in $I_a$. By standard results in the theory of distributions, we infer that there exist 
some integer $n\ge 0$ and $c_k$, $d_k\in\C$ ($k=0,\dots, n$) \st
$$
     \hat{u} = \sum_{k=0}^n c_k \delta_0^{(k)} + d_k \delta_a^{(k)},
$$
where $\delta_a$ denotes the Dirac mass at $a$. Hence, 
for other constants still labeled $c_k$ and $d_k$,
$$
     u(t) = \sum_{k=0}^n c_k t^{k} + d_k t^{k}{\rm e}^{\i at} \vp t\in\R.
$$
Since $u$ is bounded, $u(t) = c_0  + d_0{\rm e}^{\i at}$. Finally, according to 
the existence part (i) of this proof, we have
$$
      \d ( c_0  + d_0{\rm e}^{\i at}) = d_0(\i a)^\al{\rm e}^{\i at}
      \h\t{in }\S'(\R).
$$
On the other hand, since $u$ solves \eqref{3.2}, there holds
$$
      \d ( c_0  + d_0{\rm e}^{\i at}) = (\i a)^\al( c_0  + d_0{\rm e}^{\i at}).
$$
Thus, for all $t\in\R$, we have $u(t) = d_0 \exp(\i at)$.
%
\end{proof}
%
%
%
%
%
%
%
%
%
%
\section{Time fractional Schrödinger equation} Let $\al\in (0, 1)$, $d\ge 1$ be an 
integer and $A:D(A)\subseteq \ld \to\ld$ be a unbounded operator on $\ld$ with domain 
$D(A)$. For each $v\in D(A)$, we consider the following 
\emph{Time Fractional Schrödinger Equation}
%
\begin{equation}\label{4.2}
  \left\{
    \begin{aligned}
  & \mbox{\rm Find }u\in C_{\rm b}\big(\R, D(A) \big)\t{\st} \\
  & \d u = \i^\al A u,\h  \t{in }\S'(\R, \ld)\\
  & u(0) = v.
\end{aligned}\right.
\end{equation}
In \eqref{4.2}, $C_{\rm b}\big(\R, D(A) \big)$ is the space of bounded and continuous \fus defined on $\R$ with values in $D(A)$. 
%
%
\begin{rem}\label{r4.1}
There are many ways to extend the \emph{free Schrödinger equation}, i.e. 
\begin{equation} \label{4.4}
    \ddt u = -\i\Delta u,
\end{equation}
into a time fractional equation (see for instance \cite{Na}, \cite{Las}, \cite{Lu}). First, we may consider
$$
     \db^\al u = -\i\Delta u.
$$
However, as pointed out in \cite{ER2}, that equation has  
regularizing effect and dissipative properties, in a great contrast with 
\eqref{4.4}. In order to recover a hyperbolic behavior, it is better to extend \eqref{4.4} by
\begin{equation} \label{4.6}
    \db^\al u = -\i^\al \Delta u.
\end{equation}
Indeed, according to \cite[\Th 4.4 \& Example 4.6]{ER2}, this equation has no regularizing effect and possesses an asymptotic conservation law. The drawback of \eqref{4.6} regarding the dynamics, is that it does not generate a semi-group. However, \eqref{4.2} does, as we will show in the sequel.
\end{rem}
%
%
%
%
%
\begin{rem}\label{r4.2} The equation of \eqref{4.2} holds equivalently in 
$C_{\rm b}\big(\R, L^2(\rd)\big)$. In any case, $\d u $ is understood in the sense of \de \ref{d2.5}.
\end{rem}
%
%
%
%
%
%
\begin{theo}\label{t4.1} Let  $\al\in (0, 1)$, $A:D(A)\subseteq \ld \to\ld$ be a positive self adjoint operator on $\ld$ and $v\in  D(A)$. Then \eqref{4.2} has a unique \so $u$. Moreover, for all $t\in\R$,
\begin{equation} \label{4.8}
    u(t) = {\rm e}^{\i t A^{1/\al}} v,\h\t{in }D(A).
\end{equation}
\end{theo}
%
%
%

Before to prove this \th\rn, let us precise the meaning of \eqref{4.8}.  According to the \emph{Spectral \Th } in multiplication form (see \cite[\Th VIII.4]{RS} or 
\cite[\Th 10.10]{Hall}), there exist a measure space $(X,\mu)$ with finite measure $\mu$, a unitary map $\U:\ld \to \lm$ and a measurable real-valued \fu $h$ on $X$ which is finite $\mu$-a.e., \st
\begin{equation} \label{4.10}
      	\U \big(D(A) \big) = \big\{   \psi\in \lm \ | \ h\psi\in\lm \big\}
\end{equation}
and
\begin{equation} \label{4.12}
      	\U A\U^{-1}(\psi) = h\psi\vp \psi\in \U \big(D(A) \big). 
\end{equation}

Under the assumption of \Th \ref{t4.1}, we claim that
\begin{equation} \label{4.14}
      	h\ge 0 \h\mu\mbox{-a.e. on }X. 
\end{equation}
Indeed, let us denote by $(\cdot, \cdot)_{\lm}$ the standard inner product of $\lm$. We put, for each positive integer $n$,
$$
     X_n := \big\{  \xi\in X  \ | -n < h(\xi)<0  \big\},
$$
and denote by ${\rm 1}\! {\rm I}_{X_n}$ the indicator \fu of $X_n$ (i.e. ${\rm 1}\! {\rm I}_{X_n}=1$ on $X_n$ and $0$ elsewhere). Then $h{\rm 1}\! {\rm I}_{X_n}$ lies in 
$\lm$ since $\mu$ is finite. Hence
\begin{align*}  
     \big( h{\rm 1}\! {\rm I}_{X_n},  {\rm 1}\! {\rm I}_{X_n}\big)_{\lm} & =
     \big( \U A \U^{-1} ({\rm 1}\! {\rm I}_{X_n}),  {\rm 1}\! {\rm I}_{X_n}\big)_{\lm} && \mbox{\rm(by \eqref{4.12})}\\
     & = \big( A \U^{-1} ({\rm 1}\! {\rm I}_{X_n}), 
     \U^{-1}( {\rm 1}\! {\rm I}_{X_n})\big)_{\ld} 
  && \mbox{\rm (since $\U$ is unitary)}\\
     & \ge 0    
  && \mbox{\rm (since $A$ is positive).}    
\end{align*}
Thus, $\mu(X_n)=0$; and since $h$ is finite $\mu$-a.e., we obtain \eqref{4.14}.

We are now in position to define the operator ${\rm e}^{\i t A^{1/\al}}$. For each $t\in\R$, denote by $f_t$, the \fu
$$
     f_t:[0, \ii)\to \C, \h 
     x\mapsto {\rm e}^{\i t x^{1/\al}}.
$$
Then ${\rm e}^{\i t A^{1/\al}}$ is defined through \emph{bounded functional calculus} 
for self adjoint operators 
(see \cite[\Th VIII.5]{RS} or \cite[Corollary 4.43]{Kow}), by
\begin{equation} \label{4.16}
      {\rm e}^{\i t A^{1/\al}} := f_t(A) :=
      \U^{-1} M_{f_t\circ h}\U,
\end{equation}
where $ M_{f_t\circ h}:\lm\to\lm$ is the multiplication operator defined by 
$$
          M_{f_t\circ h}\psi=f_t(h)\psi\vp \psi\in\lm.
$$
%
%
%
%
\begin{proof}[Proof of \Th \ref{t4.1}](i) Uniqueness of the \so $u$. By linearity, it is enough to show that 
$v=0$ in \eqref{4.2} implies $u=0$. For each $\v\in\S(\R)$, one has
\begin{align*}  
     \U\big( \langle \d u, \v \rangle_{\S'(\R; \ld)} \big) 
     &= -\U\big( \int_\R u(t)\dm \v(t)\dt\big) 
     && \mbox{\rm(by \de\ref{d2.5})}\\
     &= -\int_\R \U\big(u(t)\big)\dm \v(t)\dt 
     && \mbox{\rm(by \cite[Prop. 1.1.6]{AB})}\\
     &= \langle \d \hat{u}, \v \rangle_{\S'(\R; \lm)} 
     && \mbox{\rm in }\lm,  
\end{align*}
according to \de \ref{d2.5} and with the notation $\hat{u}(t):=\U(u(t))$.

On the other hand,
\begin{align*}  
     \U\big( \langle\i^\al A u, \v \rangle_{\S'(\R; \ld)} \big) 
     &= \U\big( \int_\R \i^\al A u(t)\v(t)\dt\big) \\
     &= \int_\R \i^\al h\hat{u}(t)\v(t)\dt
     && \mbox{\rm(since $\U A(u(t))=h\hat{u}(t)$ by \eqref{4.12})}\\
     &= \langle \i^\al h \hat{u}, \v \rangle_{\S'(\R; \lm)}.  
\end{align*}

Since $u$ solves \eqref{4.2}, we get
$$
     \d \hat{u} =  \i^\al h \hat{u},\h\t{in }\S'\big(\R; \lm\big).
$$
In particular, for $\mu$-a.e. $\xi\in X$, there holds $h(\xi)\ge 0$ and
\begin{equation} \label{4.18}
       \d \hat{u}(\cdot, \xi) = \i^\al h(\xi) \hat{u}(\cdot, \xi),\h\t{in }\S'(\R).
\end{equation}

Besides, $\hat{u}(0, \xi)=\U(u(0))(\xi)=0$ for $\mu$-a.e.\! $\xi$ in $X$. Thus \Th \ref{t3.1} leads to $\hat{u}(\cdot, \xi)=0$; so that $u=0$. That completes the uniqueness part of the proof.
\vskip12pt
(ii) Existence. Let us show that the \fu
$$
    u:t\mapsto {\rm e}^{\i t A^{1/\al}} v
$$
solves \eqref{4.2}. For, let $t\in\R$. With the notation $\hat{w}:=\U(w)$ for $w\in\ld$, we get in view of 
\eqref{4.16}
\begin{equation} \label{4.22}
      \hat{u}(t) = f_t(h)\hat{v}.	
\end{equation}

Since $f_t$ is bounded $\mu$-a.e.\! on $X$, we deduce that $\hat{u}$ belongs to 
$L^\ii(\R; \lm)$. Then, thanks to \de \ref{d2.5}, we infer
$$
     \langle \d \hat{u}, \v \rangle_{\S'(\R; \lm)} = 
     - \int_\R \hat{u}(t) \dm \v(t)\dt,\h\t{in }\lm.
$$
Thus for $\mu$-a.e. $\xi$ in $X$,
\begin{align*}  
     \langle \d \hat{u}, \v \rangle_{\S'(\R; \lm)}(\xi) 
     &= -\int_\R f_t\big(h(\xi)\big)\hat{v}(\xi)\dm \v(t)\dt
     && \mbox{\rm(by \eqref{4.22})}\\
     &= \big\langle \d \big\{f_t\big(h(\xi)\big)\hat{v}(\xi)\big\} , 
                                 \v\big\rangle_{\S'(\R)} 
     \\
     &= \langle \i^\al h(\xi)\hat{u}(\cdot, \xi), \v \rangle_{\S'(\R)}
     && \mbox{\rm(by Th. \ref{t3.1}).} 
\end{align*}
Hence we have proved that, for $\mu$-a.e.\! $\xi$ in $X$,
\begin{equation} \label{4.26}
       \d \hat{u}(\cdot, \xi) = \i^\al h(\xi) \hat{u}(\cdot, \xi),\h\t{in }S'(\R).
\end{equation}

Moreover, by \eqref{4.22},
$$
    \n{h\hat{u}(t, \cdot) }^2_{\lm} = \int_X h^2 |\hat{v}|^2 \dif\mu <\ii,
$$
since $\hat{v}$ lies in $\U(D(A))-$ see \eqref{4.10}. Hence $t\mapsto h\hat{u}(t)$ belongs to $L^\ii(\R, \lm)$; so that, with \eqref{4.26}, we derive that
\begin{equation} \label{4.28}
       \d \hat{u} = \i^\al h \hat{u},\h\t{in }S'(\R, \lm).
\end{equation}

Let us go back in the direct space $\ld$ and write down the expected equation for $u$. In view of \eqref{4.12}, we have
$$
     h\hat{u}(t) = \U A u(t),\h\t{in }\lm.
$$
Also (see above, the uniqueness part of the proof)
$$
     \d \hat{u} = \U \big( \d u \big),\h\t{in }\S'(\R, \lm).
$$
Thus, with \eqref{4.28}
$$
     \d u = \i^\al A u ,\h\t{in }\S'(\R; \ld).
$$

Finally, $\hat{u}(0) = \hat{v}$, thus $u(0)=v$. Hence $u$ is \so to \eqref{4.2}, 
which completes the proof of the \th\rn.
\end{proof}
%
%
%

\Th \ref{t4.1} allows us to define the \emph{solution operator} associated to 
\eqref{4.2}. Indeed, under the assumptions of that \th\rn, for each $t$ in $\R$, we put
\begin{equation} \label{4.30}
      	 S_\al(t) := {\rm e}^{\i t A^{1/\al}}.
\end{equation}
Then, for each $v$ in $D(A)$, $t\mapsto S_\al(t)v$ is the \so to \eqref{4.2}, 
according to \Th \ref{t4.1}. Thus, by definition, $S_\al$ is the \emph{solution operator} associated to \eqref{4.2}. 


With these notation, we may give some properties of $S_\al$.

%
%
%
\begin{theo}\label{t4.2} Let  $\al\in (0, 1)$ and $A$ be a positive self adjoint operator on $\ld$. Then $S_\al$ defined by \eqref{4.30}, is a strongly continuous unitary group on $\ld$. Moreover, its infinitesimal generator is 
$\i A^{1/\al}$. 
\end{theo}
%
%
%
Above, $A^{1/\al}$ is defined through \emph{unbounded functional calculus}. 
More precisely, denoting by $\L(\ld)$ the space of linear and continuous maps from $\ld$ into itself, \Th 10.4 in \cite{Hall} states that there exists a unique 
\emph{projection-valued measure} $\mu^A$ on $[0, \ii)$ with values in 
$\L(\ld)$ 
\st
$$
     \int_0^\ii \la \dif\mu^A = A.
$$
Then
$$
      A^{1/\al} := \int_0^\ii \la^{1/\al} \dif\mu^A.
$$
The domain of $A^{1/\al}$ is
\begin{equation} \label{4.32}
      	D(A^{1/\al}) := \big\{ v\in\ld \ | \ 
        \int_0^\ii \la^{2/\al} \dif\mu^A_v(\la)<\ii \big\},
\end{equation}
where $\mu^A_v$ is the Borel measure on $[0, \ii)$ defined for each Borel set $E\subseteq [0, \ii)$ by
\begin{equation} \label{4.33}
      	\mu^A_v(E)= \big(v,\mu^A(E)v  )_{\ld}.
\end{equation}
We recall that $\mu^A(E)$ is an orthogonal projection of $\ld$. 
%
%
\begin{proof}[Proof of \Th \ref{t4.2}] Recalling that 
$f_t(x):={\rm e}^{\i t x^{1/\al}}$, it is clear from \eqref{4.16} that $S_\al(t)$ is bounded on $\ld$. Also, $S_\al(0)$ is equal to $Id_{\ld}$, the identity operator 
of $\ld$. Moreover, for each real numbers $t$, $s$, the properties of the 
functional calculus featured in \cite[\Th VIII.5]{RS} lead to 
$$
     S_\al(t)S_\al(s)= f_t(A)f_s(A)=f_{t+s}(A) = S_\al(t+s).
$$

Let us show that $S_\al(t)$ is unitary. For, again by functional calculus, we have
\begin{align*}  
    S_\al(t)S_\al(t)^* =  f_t(A)\overline{f_t}(A) & = |f_t|^2(A)\\
    & = Id_{\ld},
\end{align*}
since $|f_t(x)|=1$ for $x\ge 0$. In the same way, there holds 
$S_\al(t)^*S_\al(t)= Id_{\ld}$, so that $S_\al(t)$ is unitary.

Furthermore, by \eqref{4.12} and \eqref{4.16}, for each $v$ in $\ld$,
\begin{align*}  
      	\n{S_\al(t)v - v}^2_{\ld} &= \n{M_{f_t\circ h}\hat{v}- \hat{v} }^2_{\lm}\\
        &= \int_X \big| {\rm e}^{\i t h(\xi)^{1/\al}} - 1 \big|^2 |\hat{v}(\xi)|^2
                                                                  \dif\mu(\xi)
        \xrightarrow[t\to 0]{}0,                                                  
\end{align*}
by the Lebesgue dominated convergence \th\rn. Hence, $S_\al$ is strongly continuous.

Finally, let us compute its infinitesimal generator. According to 
\cite[\Th VIII.6]{RS}, there holds
$$
     S_\al(t) = \int_0^\ii f_t(\la) \dif\mu^A.
$$
Hence, for each $v$ in $D(A^{1/\al})$, \Th 5.9 in \cite[Chap 5]{Sch} leads to
$$
     \Big\| \f{S_\al(t)v - v}{t} - \i A^{1/\al}v \Big\|^2_{\ld} = 
     \int_0^\ii \Big| \f{{\rm e}^{\i t \la^{1/\al}} - 1}{t} - \i \la^{1/\al}\Big|^2
     \dif\mu^A_v(\la).
$$
Using the fact that $\la\mapsto \la^{2/\al}$ lies in $L^1(0, \ii; \mu^A_v)$ (by \eqref{4.32}), we obtain  with the dominated convergence \th\rn, that the latter integral goes to zero as $t\to 0$. That is to say $\i A^{1/\al}$ is 
the  infinitesimal generator of $S_\al$. 
\end{proof}
The next result states that the time fractional \pb \eqref{4.2} may be reduced to a time first order PDE. Notice that the latter PDE is obtained formally by raising the differential 
operators involved in \eqref{4.2} to the power $1/\al$.   
%
%
\begin{coro}  \label{c4.1} Let  $\al\in (0, 1)$, $A$ be a positive self adjoint operator on $\ld$ and $v\in D(A^{1/\al})$. Then \eqref{4.2} is equivalent to the following \pb
\begin{equation}\label{4.36}
  \left\{
    \begin{aligned}
  & \mbox{\rm Find }u\in C^1\big(\R, \ld \big)\cap 
                         C\big(\R, D(A^{1/\al}) \big)\t{\st} \\
  & \ddt u = \i A^{1/\al} u,\h  \t{in }C\big(\R, \ld \big)\\
  & u(0) = v.
\end{aligned}\right.
\end{equation}
\end{coro}
%
%
\begin{proof} If $u$ solves \eqref{4.2} then 
$ u(t) = \exp(\i t A^{1/\al}) v$ by \Th \ref{t4.1}. Then $u$ is \so to  
\eqref{4.36} according to \Th \ref{t4.2}. 
Moreover, functional calculus tells us that 
$A^{1/\al}$ is  self adjoint on $\ld$. Thus (see for instance 
\cite[Proposition 6.5]{Sch}), \eqref{4.36} has a unique \so\rn. Consequently, 
\eqref{4.2} and \eqref{4.36} are equivalent.
\end{proof}

Finally, as an example, let us consider the case
$$
    A:= -\Delta + V,
$$
where $V\in L^\ii(\rd)$ is non negative a.e.\!\! on $\rd$, and $D(A)=H^2(\rd)$. Then, for $\al\in (0, 1)$ and $v$ in $H^2(\rd)$, the time fractional Schrödinger \pb 
$$
       \d u = \i^\al(-\Delta + V)u,\h u(0) = v,
$$
has a unique \so $u$, according to \Th \ref{t4.1}. Moreover, 
\begin{equation} \label{4.42}
      	 u(t) = S_\al(t)v= {\rm e}^{\i t (-\Delta + V)^{1/\al}} v\vp t\in\R.
\end{equation}

The  requirements given in the introduction are satisfied, namely 
(i) the $L^2$-norm of $u(t)$ is conserved, i.e.
$$
     \n{u(t)}_{\ld}=\n{v}_{\ld}\vp t\in\R,
$$
and (ii) $S_\al$ is a group, i.e.
$$
    S_\al(t)S_\al(s)=S_\al(t+s)\vp  t,\, s\in\R.
$$

Also, if $v$ belongs to the domain of $(-\Delta + V)^{1/\al}$ then Corollary \ref{c4.1} yields 
that $u$ given by \eqref{4.42} is the unique \so to
$$
      \ddt u = \i (-\Delta + V)^{1/\al} u,\h   u(0) = v.
$$
\bibliography{biblio.bib}\bibliographystyle{alpha}
\vskip12pt
\end{document}